\documentclass{amsart}
\usepackage{macros}
\standardsettings

\newcommand{\NN}{\mathbb{N}}
\newcommand{\ZZ}{\mathbb{Z}}
\newcommand{\RR}{\mathbb{R}}
\newcommand{\QQ}{\mathbb{Q}}
\newcommand{\Sh}{\mathrm{Sh}}
\newcommand{\Val}{\mathrm{Val}}
\newcommand{\N}{\mathrm{N}}
\newcommand{\Len}{\mathrm{Len}}

\usepackage{listings}
\begin{document}
\title{Reduced ideals in pure cubic fields}
\author{G. Tony Jacobs}

\begin{Abstract}

Reduced ideals have been defined in the context of integer rings in quadratic number fields, and they are closely tied to the continued fraction algorithm. The notion of this type of ideal extends naturally to number fields of higher degree. In the case of pure cubic fields, generated by cube roots of integers, a convenient integral basis provides a means for identifying reduced ideals in these fields. We define integer sequences whose terms are in correspondence with some of these ideals, suggesting a generalization of continued fractions.

\end{Abstract}

\maketitle

\section{Introduction}

Quadratic fields have been studied much more extensively than their cubic analogues. Mollin, Shanks, and others developed a body of theory relating continued fractions to the ``infrastructure'' of quadratic fields, that is, to information about ideals and fractional ideals in the sub-rings of algebraic integers in quadratic fields. Hermite famously asked whether there is something which, for cubic fields, does what continued fractions do for quadratic fields.\cite{hermite} This paper is a move towards understanding infrastructure in cubic fields, and we take the existing results on the infrastructure of quadratic fields as our inspiration.

We look forward to a theory of infrastructure that applies to all cubic fields, i.e., degree-$3$ extensions of the rational numbers. We limit our work here to complex cubic fields, because such a field, in its ring of integers, has a unit group whose rank as a free $\ZZ$-module is $1$; in a totally real cubic field, the unit group has $2$ free $\ZZ$-module generators. Among the complex cubic fields, we study in particular pure cubic fields, which are generated by cube roots of integers.

A fundamental idea in infrastructure is that of a ``reduced ideal'', and in Section 2, we begin to develop this notion in the cubic case. We identify a canonical presentation for ideals, and identify necessary and sufficient conditions, in terms of that presentation, for an ideal to be reduced. We establish that each field under consideration has only finitely many reduced ideals, and we provide an algorithm for finding all of them in a given field.

In Section 3 of this paper, we construct for each field under study certain sequences, including a periodic sequence of natural numbers, which generalizes certain aspects of the continued fraction algorithm.

One appendix proves a quotidian and technical result on $\ZZ$-modules which is used when defining a canonical presentation of ideals in Section 2, and second and third appendices provide Python code, for executing algorithms described herein.

\section{Reduced ideals in pure cubic fields}

In this paper, we work in pure cubic fields, i.e., fields of the form $\QQ(\alpha)$ where $\alpha^3=m$ for some positive cube-free integer $m$. Before proceeding at this level of specificity, we remark briefly on why this is our chosen purview.

We know from Dirichlet's unit theorem (see e.g. \cite[p. 346]{alacawilliams}) that a cubic field with three real embeddings (equivalently, positive discriminant) has a unit group of rank $2$, whereas a field with one real embedding and two complex embeddings (equivalently, negative discriminant) has a unit group of rank $1$. This makes those with complex embeddings, such as pure cubic fields, easier to study, especially when generalizing notions from real quadratic fields, which also have rank $1$ unit groups.

There is no loss of generality in choosing $m>0$, because $\QQ(\sqrt[3]{-m})=\QQ(\sqrt[3]{m})$.

Besides the structure of the unit group, we are working with the integral basis of each field under consideration. Among the cubic fields of negative discriminant, those of the form we study here have a particularly simple integral basis. A clear next step in generalization would be to address arbitrary cubic fields with negative discriminants; an integral basis of such fields is given, for example, in \cite{albert} and more briefly in \cite{alaca}. Another possible generalization would be to arbitrary fields with rank one unit groups, including imaginary quartic fields, as in \cite{buchmannwilliams}.

There are certain notions present in this paper, such as the ``shadow'' of an algebraic number, that would seem to generalize to number fields of arbitrary degree, with arbitrary unit group structure.

\subsection{Basic Results on Ideals}
We begin with some notation. Recall that a free $\ZZ$-module $M$ of rank $n$ is defined as a free abelian group on $n$ generators, with the operation \textit{scalar multiplication}, defined so that for $m\in M$ and for the scalar $n > 0 \in \ZZ$, we have $nm = (n-1)m + m$, $0m = 0$, and $(-n)m = -(nm)$.

\begin{remark}
\label{notation}
Throughout this work, we adopt the following notations: The free $\ZZ$-module $u_1\ZZ\oplus\cdots\oplus u_n\ZZ$ will be denoted $[u_1,\ldots,u_n]$. The greatest common divisor of the two integers $i$ and $j$ will be denoted $(i,j)$. Furthermore, we define the following variables for use throughout:

Let $h$ and $k$ be relatively prime, squarefree positive integers. Then $m=hk^2$ is a positive cube-free integer, and we note that every such integer can be expressed uniquely in this form. Set $\sigma=3$ if $m\equiv\pm 1\pmod{9}$, and $1$ otherwise. Let $\alpha$ be the unique real root of the $\QQ$-irreducible polynomial $x^3-m$; the field $K=\QQ(\alpha)$ is a degree $3$ extension of $\QQ$. Setting $\widehat{m}=\widehat{\alpha}^3=h^2k$, we obtain another expression for the same field: $\QQ(\widehat{\alpha})=\QQ(\alpha)$. We note that $\sigma$ is invariant under switching $h$ and $k$; to avoid redundancy, we adopt the convention that $h>k$.

The field $K$ has integral basis $\left\{1,\alpha,\theta=\frac{1}{\sigma}\left(k \pm k\alpha + \widehat{\alpha}\right)\right\}$ \cite[p.176]{alacawilliams}. In cases where $\sigma=1$, it doesn't matter which of the plus or minus is chosen; in our calculations, we use the plus in these cases.
\end{remark}

We consider $\mathcal{O}_K$, the ring of integers of $K$, as the free $\ZZ$-module $[1,\alpha,\theta]$. Recall that a $\ZZ$-submodule is a subset of a $\ZZ$-module that is, itself, a $\ZZ$-module. Now, a non-zero ideal in $\mathcal{O}_K$ is necessarily a $\ZZ$-submodule of full rank, but not every rank-$3$ $\ZZ$-submodule of $\mathcal{O}_K$ is an ideal. Our first results concern the understanding of ideals in terms of their structure as $\ZZ$-submodules. We begin with a quotidian but useful fact about free $\ZZ$-modules, the proof of which is elementary and found in the appendix.

\begin{lemma}
\label{modulelemma}
Let $M=[u_1,\ldots,u_n]$ be a free $\ZZ$-module of rank $n$. Let $M'\subseteq M$ be a submodule of full rank. Then we can write $M'=[a_{1,1}u_1, \ldots, a_{n,1}u_1+\cdots+a_{n,n}u_n]$, with all coefficients integral. Furthermore, we can suppose without loss of generality that, for $i=1,\ldots,n$, we have $a_{i,i}$ strictly positive, and for $j=i+1,\ldots,n$ we have $0\leq a_{j,i}<a_{i,i}$. Subject to these conditions, all $\frac{n(n+1)}{2}$coefficients are uniquely determined.
\end{lemma}

\begin{definition}
Let $I$ be a submodule of the ring of integers in $K=\QQ(\alpha)$, with $m,h,k,\alpha,\text{ and } \theta$ as above, and per lemma \ref{modulelemma}, let $I$ be written in the form $I=[a,b+c\alpha,d+e\alpha+f\theta]$ with $a,c,f>0$, $0\leq b<a$, $0\leq d<a$, and $0\leq e<c$. We refer to this expression as \textit{canonical form} for the submodule. The product $\N(I)=acf$ is uniquely determined by canonical form, and we define the \textit{norm of the submodule} to be this number. The smallest rational integer in the submodule, given by the number $a$ in canonical form, is defined as the \emph{length of the submodule}, sometimes denoted $\Len(I)$, and we will sometimes write $L$ instead of $a$. 
\end{definition}

We note that the norm we have defined is precisely the index of the $\ZZ$-submodule in the $\ZZ$-module $\mathcal{O}_K$. Thus, in cases where the submodule is an ideal, the norm is the same as the ideal norm, defined in the usual way \cite[p.221]{alacawilliams}. In this case, we may refer to the submodule's norm, length, and canonical form as the norm, length, and canonical form of the ideal. We will now establish a proposition that gives us a way to determine, from its canonical form, when a submodule of $\mathcal{O}_K$ is an ideal. We first prove a technical lemma that will simplify our notation:

\begin{lemma}
\label{pqrstlemma}
Let $m, h, k, \sigma \textit{ and } \pm$ be as in Remark \ref{notation}.
Then the numbers:

\begin{enumerate}
	\item $p=\frac{1}{\sigma}(hk\mp k^3)$
	\item $q=\frac{1}{\sigma}(k-k^3)$
	\item $r=\frac{1}{\sigma^2}(k^2\mp 2h + 1)$
	\item $s=\frac{1}{\sigma^2}(h\mp k^4)$
	\item $t=\frac{1}{\sigma}(k^3 + 2k)$
\end{enumerate}

are all integers.
\end{lemma}

\begin{proof}
In cases where $\sigma=1$ there is nothing to show. Thus, assume that $m=hk^2\equiv\pm 1\pmod{9}$, which implies that $h\equiv\pm 1\pmod{3}$, hence that $h^3\equiv\pm 1\pmod{9}$, and also that $k^2\equiv h^2\pmod{9}$. From these facts, we obtain the following:

\begin{enumerate}
	\item $3p = hk\mp k^3 = k(h\mp k^2) \equiv k(h\mp 1) \equiv 0\pmod 3$
	\item $3q = k-k^3 = k(1-k^2) \equiv k(0) \equiv 0\pmod{3}$
	\item $9r = k^2\mp 2h + 1 \equiv h^2\mp 2h + 1 \equiv (h\mp1)^2 \equiv 0\pmod{9}$
	\item $9s = h\mp k^4 \equiv h\mp h^4 \equiv h(1 \mp h^3) \equiv 0\pmod 9$
	\item $3t = k^3+2k = 3k - 3q \equiv 0\pmod{3}$
\end{enumerate}
\end{proof}

We now state our condition for a submodule to be an ideal. A note of perspective: This theorem is a list of $16$ divisibility conditions; in the corresponding theorem on quadratic fields, the number of conditions is $3$ \cite[p.9]{mollin}. In generalizing these results to broader classes of number fields, characterizing ideals in terms similar to these seems likely to become tortuously technical, assuming that the underlying theory and approach remain the same. 

\begin{proposition}[Identification of ideals]
\label{IDideal}
Take $m, h, k, \sigma, \pm, \alpha, \theta$ and $K$ as in Remark \ref{notation}, take $p, q, r, s$ and $t$ as in Lemma \ref{pqrstlemma}, and let $M=[a,b+c\alpha,d+e\alpha+f\theta]$ be a submodule of the $\ZZ$-module $\mathcal{O}_K$, in canonical form. Then $M$ is an ideal if and only if the following divisibility conditions are all satisfied:

\begin{enumerate}
	\item
	\begin{enumerate}
		\item $c|a$
		\item $c|b$
	\end{enumerate}
	\item
	\begin{enumerate}
		\item $f|a$
		\item $f|\sigma kc$
		\item $f|\sigma ke$
		\item $f|b \pm k^2c$
		\item $f|d \pm k^2e$
	\end{enumerate}
	\item
	\begin{enumerate}
		\item $cf|ae$
		\item $cf|be - cd$
		\item $cf|be \pm k^2ce$
		\item $cf|df + qf^2 - \sigma ke^2 \mp 2k^2ef$
		\item $cf|qef + sf^2 - de - tef \mp k^2e^2$
	\end{enumerate}
	\item
	\begin{enumerate}
		\item $acf|(k^2c^2 + b^2)f \mp k^2bcf - \sigma kc(be - cd)$
		\item $acf|(pc - qb)cf + (b \pm k^2c)(be - cd)$
		\item $acf|(pcf - k^2ce - bd - qbf)f \pm k^2f(2be - cd) + \sigma ke(be - cd)$
		\item $acf|(pce - rk^2cf - qbe - sbf)f +(d + tf \pm k^2e)(be - cd)$
	\end{enumerate}
\end{enumerate}
\end{proposition}

\begin{proof}
In order for $M$ to be an ideal, it must contain the products of $\alpha$ and $\theta$ with each of its $\ZZ$-generators. Thus, each of these products must be an integer combination of the generators. Multiplication by $\alpha$ and by $\theta$ are represented by the matrices:

\begin{equation*}
P_1 = \left[\begin{matrix} 0 & -k^2 & p \\ 1 & \mp k^2 & q \\ 0 & \sigma k & \pm k^2 \end{matrix}\right] \text{ and } P_2 = \left[\begin{matrix} 0 & p & -k^2r \\ 0 & q & s \\ 1 & \pm k^2 & t \end{matrix}\right],
\end{equation*}

respectively, with respect to our integral basis $\{1,\alpha,\theta\}$. We place the generators of $M$ in a matrix $A$, and we define a matrix $Q$ representing row operations that reduce $A$ as follows:

\begin{equation*}
A=\left[\begin{matrix} a & b & d \\  & c & e \\  &  & f \end{matrix}\right],\,\, QA = \left[\begin{matrix}acf &  &  \\  & cf &  \\  &  & f \end{matrix}\right]
\end{equation*}

We note that $Q$ is an integer matrix, invertible over $\QQ$, but not, in general, over $\ZZ$. Since we want the columns of $P_1A$ and $P_2A$ to be integer combinations of the columns of $A$, we use row reduction on the augmented matrix:

\begin{equation*}
\left[\begin{array}{c|cc} A & P_1A & P_2A \end{array}\right] \sim \left[\begin{array}{c|cc} QA & QP_1A & QP_2A \end{array}\right],
\end{equation*}

and check that the columns of $QP_1A$ and $QP_2A$ are integer combinations of the columns of $QA$. Since these products are unwieldy, we show results for the first two columns of $A$, and then the third:

\begin{align*}
QP_1\left[\begin{matrix} a & b \\ 0 & c \\ 0 & 0 \end{matrix}\right] &= \left[\begin{matrix} abf & -(k^2c^2 + b^2)f \pm k^2bcf + \sigma kc(be - cd) \\ af & bf \mp k^2cf - \sigma kce \\ 0 & \sigma kc \end{matrix}\right] \\
QP_2\left[\begin{matrix} a & b \\ 0 & c \\ 0 & 0 \end{matrix}\right] &= \left[\begin{matrix} a(be-cd) & (pc - qb)cf + (b \pm k^2c)(be - cd) \\ -ae & qcf - be \mp k^2ce \\ a & b \pm k^2c \end{matrix}\right] \\
QP_1\left[\begin{matrix} d \\ e \\ f \end{matrix}\right] &= \left[\begin{matrix} (pcf - k^2ce - bd - qbf)f \pm k^2f(2be - cd) + \sigma ke(be - cd) \\ df \mp k^2ef + qf^2 - \sigma ke^2 \mp k^2ef \\ \sigma ke \pm k^2f \end{matrix}\right] \\
QP_2\left[\begin{matrix} d \\ e \\ f \end{matrix}\right] &= \left[\begin{matrix} (pce - rk^2cf - qbe - sbf)f +(d + tf \pm k^2e)(be - cd) \\ qef + sf^2 - de \mp k^2e^2 - tef \\ d \pm k^2e + tf \end{matrix}\right]
\end{align*}

Again, in order for $M$ to be an ideal, all of the columns on the right sides of these equations must be integer combinations of the columns of $QA$. It is thus clear that our list of divisibility conditions is both necessary and sufficient for our result. This completes the proof.
\end{proof}

Using this result, we can fix $m$, thus choosing a field, and write down all ideals of a given length $L$. Recalling that the ideal $[a, b + c\alpha, d + e\alpha + f\theta]$ is called primitive if the greatest common divisor of the integers $a,b,c,d,e,f$ is $1$, it is also possible to list all primitive ideals with length $L$. We have written such an algorithm in Python, and it is available online.

\subsection{Reduced Ideals}
Next, we generalize the definition of a reduced ideal in a quadratic field (\cite[p.19]{mollin}) to fields of arbitrary degree. First, we define the ``shadow'' of a number.

\begin{definition}
Let $\beta$ be an algebraic number in a number field $K$, a finite extension of $\QQ$. Then we define the shadow of $\beta$, $\Sh(\beta)=\Sh_K(\beta)$, as the product of all of its algebraic conjugates for that field, excluding itself.
\end{definition} 

If $K$ is a quadratic extension of $\QQ$ and $\beta$ is irrational, then $\Sh(\beta)$ is simply the algebraic conjugate of $\beta$. If $K$ is a degree $n$ extension, and $\beta$ is rational, then $\Sh(\beta)=\beta^{n-1}$. In any case, we have that $\Sh(\beta)\cdot\beta=\N(\beta)$, where $\N(\beta)$ represents the usual norm of an algebraic number in a number field. This gives us that $\Sh(\beta)\in K$; we also note that if $\beta$ is an algebraic integer, then $\Sh(\beta)$ is also an algebraic integer.

We give two formulas for the shadow of a number, when that number is given in terms of our known integral basis for pure cubic fields:

\begin{proposition}
\label{ShadowProp}
Take $m, h, k, \sigma, \pm, \alpha, \widehat{\alpha}, \theta \textit{ and } K$ as in Remark \ref{notation}, and let $\beta \in K$, so $\beta = x + y \alpha + z \theta$ for some rational $x,y,z$. Set:

\begin{equation*}
\left[\begin{array}{c} \widetilde{x} \\[1mm] \widetilde{y} \\[1mm] \widetilde{z} \end{array}\right] = \left[\begin{array}{c} x + \frac{zk}{\sigma} \\[1mm] y \pm \frac{zk}{\sigma} \\[1mm] \frac{z}{\sigma} \end{array}\right] = \left[\begin{array}{ccc} 1 &  & \frac{k}{\sigma} \\[1mm]  & 1 & \pm\frac{k}{\sigma} \\[1mm]  &  & \frac{1}{\sigma} \end{array}\right]\left[\begin{array}{c} x \\[1mm] y \\[1mm] z \end{array}\right]
\end{equation*}

Then we have:

\begin{align*}
\Sh(\beta) &= \widetilde{x}^2 + \alpha^2\widetilde{y}^2 + \widehat{\alpha}^2\widetilde{z}^2 - \alpha\widetilde{x}\widetilde{y} - \widehat{\alpha}\widetilde{x}\widetilde{z} - \alpha\widehat{\alpha}\widetilde{y}\widetilde{z} \\
&= (\widetilde{x}-\widehat{\alpha}\widetilde{z})^2 -\alpha(\widetilde{x}-\widehat{\alpha}\widetilde{z})(\widetilde{y}-\tfrac{\widehat\alpha}{\alpha}\widetilde{z}) + \alpha^2(\widetilde{y}-\tfrac{\widehat{\alpha}}{\alpha}\widetilde{z})^2.
\end{align*}

Also, $\Sh(\beta)\geq 0$ for all $\beta\in K$.
\end{proposition}

\begin{proof}
Note that we can write $\beta = x + y\alpha + z\theta = \widetilde{x} + \widetilde{y}\alpha + \widetilde{z}\widehat{\alpha}$.

Now let $\omega=e^{2\pi i/3}$. The algebraic conjugates of $\alpha$ are $\alpha\omega$ and $\alpha\omega^2$; the corresponding conjugates of $\widehat{\alpha}=\tfrac{\alpha^2}{k}$ are $\widehat{\alpha}\omega^2$ and $\widehat{\alpha}\omega$. Thus the algebraic conjugates of $\beta$ are $\widetilde{x}+\widetilde{y}\alpha\omega + \widetilde{z}\widehat\alpha\omega^2$ and $\widetilde{x}+\widetilde{y}\alpha\omega^2 + \widetilde{z}\widehat\alpha\omega$. Mutiplying these expressions together, we obtain our first formula.

Observing the form of each term, we note that, considered as an equation in $\RR^3$, our first formula vanishes along the line $\widetilde{x}=\alpha\widetilde{y}=\widehat{\alpha}\widetilde{z}$. We therefore rewrite it in terms of the displacements $\widetilde{x}-\widehat{\alpha}\widetilde{z}$ and $\widetilde{y}-\tfrac{\widehat{\alpha}}{\alpha}\widetilde{z}$, and we have the second formula.

Finally, it is clear from the second formula that the function $\Sh$ is a positive definite quadratic form in the variables $(\widetilde{x} - \widehat{\alpha}\widetilde{z})$ and $(\widetilde{y}-\tfrac{\widehat{\alpha}}{\alpha}\widetilde{z})$.
\end{proof}

 We are now ready to give a general definition of a reduced ideal.

\begin{definition}
Let $K$ be a degree $n$ extension of the rationals, let $I$ be a primitive ideal in its ring of integers, and let $L=\Len(I)$. We define $I$ to be a \emph{reduced} ideal if for all $\beta\in I$, the pair of inequalities $|\beta|<L$ and $|\Sh(\beta)|<L^{n-1}$ together imply that $\beta=0$.
\end{definition}

In our case, with $K=\QQ(\alpha)$, where $\alpha=\sqrt[3]{m}$, we have an explicit description of ideals, and we can determine whether an ideal is reduced by examining its canonical form. The following three results are inspired by Mollin's Theorem 1.4.1 and its corollaries in \cite{mollin}.

\begin{theorem}[Identification of reduced ideals]
\label{IDreduced}
Let $I=[L,b+c\alpha,d+e\alpha+f\theta]$ be the canonical form of a primitive ideal in $\mathcal{O}_K$, with $m, h, k, \sigma, \pm, \alpha, \widehat{\alpha}, \theta \text{ and } K$ as previously. Then $I$ is reduced if and only if for every integer pair $(y,z)\neq (0,0)$ satisfying:

\begin{enumerate}
	\item $0\leq z<\frac{\sigma L}{\widehat{\alpha}}$
	\item $f|z$
	\item $\left(\frac{\widehat{\alpha}}{\alpha}\mp k\right)\frac{z}{\sigma} - \frac{2L}{\sqrt{3}\alpha} \leq y \leq \frac{\sigma L - \widehat{\alpha}z + \sqrt{(\sigma L - \widehat{\alpha}z)(\sigma L + 3\widehat{\alpha}z)}}{2\sigma\alpha} \mp \frac{kz}{\sigma}$
	\item $c|y-\frac{ez}{f}$,
\end{enumerate}

we have the inequality:

\begin{equation*}
\left\lfloor\frac{Q-(\frac{ybf+zcd-zbe}{cf})}{L}\right\rfloor<\frac{P-(\frac{ybf+zcd-zbe}{cf})}{L}, 
\end{equation*}

where

\begin{align*}
P &= P(y,z) \\
&=\max\left\{\frac{-\sigma L - kz - \left(\alpha(\sigma y \pm kz) + \widehat{\alpha}z\right)}{\sigma},\frac{\sigma\alpha y \pm \alpha kz + \widehat{\alpha}z - 2kz - \sqrt{(2\sigma L)^2 - 3(\sigma\alpha y \pm \alpha kz - \widehat{\alpha}z)^2}}{2\sigma} \right\},
\end{align*}

and 

\begin{align*}
Q &= Q(y,z) \\
&=\min\left\{\frac{\sigma L - kz - \left(\alpha(\sigma y \pm kz) + \widehat{\alpha}z\right)}{\sigma},\frac{\sigma\alpha y \pm \alpha kz + \widehat{\alpha}z - 2kz + \sqrt{(2\sigma L)^2 - 3(\sigma\alpha y \pm \alpha kz - \widehat{\alpha}z)^2}}{2\sigma} \right\}.
\end{align*}
\end{theorem}

\begin{proof}
Let $\phi:\mathcal{O}_K\to \RR^3$ be the additive homomorphism defined by $1\mapsto(1,0,0)$, $\alpha\mapsto(0,1,0)$ and $\theta\mapsto(0,0,1)$. This map is an isomorphism of the $\ZZ$-modules $\mathcal{O}_K$ and $\ZZ^3$, the latter of which is embedded in $\RR^3$.

Now, let $\beta\in I\subset\mathcal{O}_K$. Our condition that $|\beta|<\Len(I)$ transforms into the geometric condition that the point $\phi(\beta)$ lies strictly between two planes: $\widetilde{x}+\alpha\widetilde{y}+\widehat{\alpha}\widetilde{z}=\pm L$. (Here we take $\widetilde{x}, \widetilde{y}, \widetilde{z}$ as in \ref{ShadowProp}.) Our second condition, that $\left|\Sh(\beta)\right|<\Len(I)^2$, transforms into the geometric condition that the point $\phi(\beta)$ lies in the interior of the oblique elliptic cylinder given by $(\widetilde{x} - \widehat{\alpha}\widetilde{z})^2 - \alpha(\widetilde{x} - \widehat{\alpha}\widetilde{z})(\widetilde{y} - \tfrac{\widehat\alpha}{\alpha}\widetilde{z}) + \alpha^2(\widetilde{y} - \tfrac{\widehat{\alpha}}{\alpha}\widetilde{z})^2 = L^2$.

These conditions define an open region $R$, between two planes and inside an elliptic cylinder, which is bounded and symmetric about the origin; it contains the images of $0$ and of at most finitely many other elements in the ideal $I$. The ideal is reduced if and only if $\phi(\beta)\not\in R$ for every non-zero $\beta\in I$. It is thus sufficient to write conditions establishing that the interior of $R$ contains no $\phi$-images of non-zero ideal elements.

First, we can ignore the line $y=z=0$, which contains images of rationals, because it intersects the boundaries of $R$ at $(\pm L,0,0)$, and there are no non-zero rational ideal elements between these two points.

Now, the entire region $R$ satisfies $|z|\leq\sup\{z : z\in R\} = \frac{\sigma L}{\widehat{\alpha}}$, because this is the maximum $z-$coordinate of the intersection of our elliptic cylinder with either plane. By symmetry, and because images of ideal elements have integer coordinates in $x, y, z$, we need only consider integer values from $z=0$ to $z=\left\lfloor\frac{\sigma L}{\widehat{\alpha}}\right\rfloor$. Furthermore, images of ideal elements will have $z$-coordinates that are multiples of $f$. Thus we obtain conditions (1) and (2) of this theorem.

For each integer $z$ in that range, we can bound possible $y$-values with the inequalities given in condition (3). The lower bound is the minimum $y$-value attained by a point on the elliptic cylinder, and the upper bound is the maximum $y$-value of a point of intersection of the cylinder and the planes. The final condition, that $y-\frac{ez}{f}\equiv 0\pmod{c}$, simply restricts our checking to $y$ values where images of ideal elements occur.

For each $(y,z)$ pair that we check, we wish to verify that no image of an ideal element lies in region $R$ along the line corresponding to those $y$ and $z$ values. This can be expressed by saying that the first image to the left of the right edge of $R$ is also to the left of the left edge of $R$. We write that geometric condition algebraically as the inequality $\left\lfloor\frac{Q-(\frac{ybf+zcd-zbe}{cf})}{L}\right\rfloor<\frac{P-(\frac{ybf+zcd-zbe}{cf})}{L}$, where $P=P(y,z)=\inf\{x:(x,y,z)\in R\}$ and $Q=Q(y,z)=\sup\{x:(x,y,z)\in R\}$ are the $x$-coordinates of the left and right edges of $R$, respectively. This proves the theorem.
\end{proof}

Using the same geometric construction (and the same notation) from this proof, we establish the following corollaries:

\begin{corollary}[Lower bound]
\label{LB}
If $L\leq\min\{\alpha,\frac{\widehat{\alpha}}{\sigma}\}$, then $I$ is reduced.
\end{corollary}

\begin{proof}
Since $L\leq\frac{\widehat{\alpha}}{\sigma}$, then in the above theorem, the only $z$-value satisfying inequality (1) is $z=0$. The intersection of $R$ with the plane $z=0$ has maximum/minimum $y$-vales of $\pm \frac{L}{\alpha}$, so with $L<\alpha$ the only integer $y$-value in our region is $y=0$. As noted in the proof of the theorem, no images of non-zero ideal elements are found in the interior of $R$ along the line $y=z=0$.
\end{proof}

\begin{lemma}
\label{fbound}
Using the notation from the theorem, if $I=[L,b+c\alpha,d+e\alpha+f\theta]$ is a primitive ideal, then $f|\sigma k$.
\end{lemma}

\begin{proof}
Let $g=(f,k)$, so we can write $f=gf'$ and $k=gk'$ with $(f',k')=1$. Then it follows from the ideal conditions $f|\sigma kc$ and $f|\sigma ke$ that $f'|\sigma c$ and $f'|\sigma e$, respectively. In the case where $\sigma=3$ and $3|f'$, let $f''=\frac{f'}{3}$, else let $f''=f'$. Then we have that $f''$ divides $a,c,e \text{ and } f$. Examining the conditions $f|b \pm k^2c$ and $f|d \pm k^2e$, we see that $f''$ also divides $b \text{ and } d$.

Since I is primitive, these divisibility conditions give us that $f''=1$, so $f=g$ or $f=3g$, the latter only if $\sigma=3$ and $3|f'$. In either case, $f|\sigma k$, as claimed.
\end{proof}

\begin{theorem}[Upper bound]
\label{UB}
If $L>\frac{6\sqrt{3}m}{\pi}$, then $I$ is not reduced.
\end{theorem}

\begin{proof}
The region $R$ in Theorem \ref{IDreduced} is convex and symmetric about the origin, and we claim its volume is equal to $\frac{4\pi \sigma L^3}{3\sqrt{3}\alpha\widehat{\alpha}}$. First, in the basis $\{\widetilde{x},\widetilde{y},\widetilde{z}\}$ the perpendicular distance between the planes $\widetilde{x}+\alpha \widetilde{y}+\widehat{\alpha}\widetilde{z}=\pm L$ is $\frac{2L}{\sqrt{1+\alpha^2+\widehat{\alpha}^2}}$, and the area of the ellipse at each end of $R$ is $\frac{2\pi L^2\sqrt{1+\alpha^2+\widehat{\alpha}^2}}{3\sqrt{3}\alpha\widehat{\alpha}}$. That gives us a volume of $\frac{4\pi L^3}{3\sqrt{3}\alpha\widehat{\alpha}}$. Translating back to the basis $\{x,y,z\}$, we pick up a factor of $\sigma$.

It now follows from Minkowski's convex body theorem (see, e.g. page 306 in \cite{alacawilliams}) that if $I$ is reduced, then $\N(I)\geq \frac{\pi\sigma L^3}{6\sqrt{3}\alpha\widehat{\alpha}}$, or $\frac18$ the volume of $R$. On the other hand, since $I$ is primitive if it is reduced, we also have from our canonical form and from Lemma \ref{fbound} that $\N(I)=acf\leq \sigma kL^2$. These two inequalities are incompatible for $L>\frac{6\sqrt{3}m}{\pi}$, so we have our result.
\end{proof}

We note a lack of symmetry in this upper bound formula. Since the field $K$ is generated indifferently by $\sqrt[3]{m}$ or $\sqrt[3]{\widehat{m}}$ (where $m=hk^2$ and $\widehat{m}=h^2k$), it seems odd that our upper bound includes one or the other, and isn't simply in terms of $h$ and $k$. Indeed, if we swap $h$ and $k$, we would find that ideals cannot be reduced with length greater than $\frac{6\sqrt{3}\widehat{m}}{\pi}$, but since $m<\widehat{m}$, the stated result is stronger.

\begin{theorem}
\label{finiteness}
Let $K=\mathbb{Q}(\alpha)$ where $\alpha^3=m$, for a cube-free integer $m$. Then the ring $\mathcal{O}_K$ contains at least one, and only finitely many, reduced ideals.
\end{theorem}

\begin{proof}
The entire ring $\mathcal{O}_K$ is always a reduced ideal, so we have at least one. By the above theorem, the length $L$ of a reduced ideal is bounded, say $L<L_0$. Thus, in accordance with the observation made in our proof of Theorem \ref{UB}, its norm is also bounded, by $\sigma k L_0^2$. Since there are only finitely many ideals of a given norm, \cite[p.313]{alacawilliams} we have this result as well.
\end{proof}

The above results (\ref{IDreduced} - \ref{UB}) give us a way of efficiently computing a complete list of reduced ideals in the fields we have been studying. We check for them by examining ideals of each length less than the upper bound of Theorem \ref{UB}. For each length, we produce a list of ideals, per the remarks following Proposition \ref{IDideal}.

As long as the length is less than $\min\left\{\alpha,\frac{\widehat{\alpha}}{\sigma}\right\}$, each ideal of that length is necessarily reduced by Corollary \ref{LB}. For each length between this minimum value and the upper bound, we examine each ideal. For each one, we obtain a list of pairs $(y,z)$ satisfying conditions (1), (2), (3) and (4). For each such pair, we calculate $P$ and $Q$ and check our main condition from Theorem \ref{IDreduced}. In Appendix 2, we provide Python code that executes this algorithm.

(In the code, for the sake of efficiency, we are able to skip the calculation for some ideals that, based on their canonical form, cannot be reduced. In short, if either coefficient $c$ or $f$ is too small, then Minkowski's convex body theorem makes it impossible for the ideal to be reduced.)

The following definition affords a different characterization of reduced ideals which will prove useful.

\begin{definition}
Let $I$ be an ideal (or fractional ideal) in a number field. Then $\beta\in I$ is a \emph{minimal element of I} if $\left|\gamma\right|<\left|\beta\right|$ and $\left|\Sh(\gamma)\right|<\left|\Sh(\beta)\right|$ for $\gamma\in I$ together imply that $\gamma=0$.
\end{definition}

Now we can characterize reduced ideals in terms of minimal elements.

\begin{theorem}
Let $I$ be an ideal in a number field. Then $I$ is reduced if and only if there is some rational $q\in I$ that is a minimal element of $I$.
\end{theorem}

The proof is immediate from the definition. In particular, if $I$ is reduced, then $q=\pm L(I)$ is a minimal element, and conversely.

\section{Periodic norm sequences}

As seen in Mollin's \textit{Quadratics}, the terms in a quadratic number's continued fraction expansion can be put in correspondence with a sequence of ideals, and the eventual periodicity of these sequences corresponds to the presence of finitely many reduced ideals in an equivalence class \cite[p. 44]{mollin}. We now develop a corresponding notion for a class of cubic numbers.

Throughout this section, let $h$, $k$, $\sigma$, $m$, $\alpha$, $\theta$, $\widehat\alpha$, $K$, and $\pm$ be as in Remark \ref{notation}, and for any triple $(x,y,z)$ define $(\widetilde{x},\widetilde{y},\widetilde{z})$ as in Proposition \ref{ShadowProp}. Let $\phi$ be the additive homomorphism defined in the proof of Theorem \ref{IDreduced}.
 
Define the functions $\Val,\Sh:\RR^3\to\RR$ by the formulas $\Val(x,y,z)= x+ \alpha y + \theta z$ and $\Sh(x,y,z)=(\widetilde{x}-\widehat\alpha\widetilde{z})^2 - (\widetilde{x}-\widehat\alpha\widetilde{z})(\alpha\widetilde{y}-\widehat\alpha\widetilde{z}) + (\alpha\widetilde{y}-\widehat\alpha\widetilde{z})^2$. Also define the function $\N(x,y,z)=\Sh(x,y,z)\Val(x,y,z)$. Then, for $\beta\in K$, we have $\Val(\phi(\beta))=\beta$, $\Sh(\phi(\beta))=\Sh(\beta)$ and $N(\phi(\beta))=N(\beta)$. Furthermore, if $(x,y,z)\in\QQ^3$, then $\Val(x,y,z)=\phi^{-1}(x,y,z)$, $\Sh(x,y,z)=\Sh(\phi^{-1}(x,y,z))$, and $\N(x,y,z)=\N(\phi^{-1}(x,y,z))$.

Taking $a$ and $b$ positive, define the region:

\begin{equation*}
R_{a,b}=\left\{(x,y,z)\in\RR^3 \,:\, \left|\Val(x,y,z)\right|<a,\,\Sh(x,y,z)<b \right\}.
\end{equation*}

This region is convex and symmetric about the origin. (In this notation, the region examined in Theorem \ref{IDreduced} is $R_{L,L^2}$.) We now use regions of this form to define, associated with $\alpha$, a sequence $\left(\beta_n\right)_{n\geq 0}$ of algebraic numbers, and a sequence $\left(N_n\right)_{n\geq 0}$ of integers.

Let $\beta_0=1$, and $P_0=\phi(\beta_0)=(1,0,0)$. We begin with $R_{a_0,b_0}=R_{1,1}$, a region with the point $P_0$ on its boundary, and with no non-zero lattice points in its interior. To find $P_{n+1}$, let $a_{n+1}$ be the maximum positive number such that $R_{a_{n+1},b_n}$ has no lattice point in its interior. Such a number is guaranteed by Minkowski's convex body theorem. We will actually encounter two lattice points at once, because of symmetry; take $P_{n+1}$ to be the one for which the function $\Val$ is positive. We have $a_{n+1}=\Val(P_{n+1})$; also set $b_{n+1}=\Sh(P_{n+1})$, and let $\beta_{n+1}=\phi^{-1}(P_{n+1})$.

\begin{definition}
The sequence $\left(\beta_n\right)_{n\geq 0}$ is the \emph{minimal sequence associated with $\alpha$}, and $\left(\N_n\right)_{n\geq 0}=(\N(P_n))$ is the \emph{norm sequence associated with $\alpha$}.
\end{definition}

We have Python code online that calculates the minimal sequence and norm sequence of $\alpha$ given an appropriate value for $m$.

We note that the minimal sequence of $\alpha$ is precisely the sequence of minimal elements of $\mathcal{O}_K$, starting with $\beta_0=1$ and proceeding through minimal elements in order of increasing absolute value. The algorithm could be modified to run backwards, by holding cylinder heights constant and increasing their widths to find new points. This would give us the rest of the positive minimal elements, those with absolute values between $0$ and $1$. However, as we shall see, the sequence we have defined contains all the information we need. We first note some useful facts:

\begin{proposition}
If $\beta$ is a minimal element in the ideal (or fractional ideal) $I$, and $\gamma$ is another field element, then $\gamma\beta$ is a minimal element in the ideal (or fractional ideal) $(\gamma)\cdot I$.
\end{proposition}

\begin{proof}
This follows immediately because the functions $\Sh:K\to\RR$ and $|\cdot|:K\to\RR$ are both multiplicative.
\end{proof}

\begin{remark}[Dirichlet's Unit Theorem]
The unit group of $K$, a number field of degree $3$ with one real embedding and one pair of complex embeddings (i.e., a cubic field with negative discriminant), is of the form $U_K=\{\pm\varepsilon_0^j | j\in\ZZ\}$, where $\varepsilon_0\in K$ is the fundamental unit of the number field, which satisfies $\varepsilon_0>1$. (See, e.g., \cite[p.346,p.362]{alacawilliams}.)
\end{remark}

Now we are ready to show that the norm sequence we have defined is indeed periodic.

\begin{theorem}
\label{periodic}
The norm sequence of $\alpha$ is periodic, and the minimal sequence has the property that $\beta_{i+l}=\varepsilon_0\beta_i$, where $l$ is the period of the norm sequence, and $\varepsilon_0$ is the fundamental unit of the field $K=\QQ(\alpha)$.
\end{theorem}

We note that this theorem is closely analogous to Proposition 2.6 from \cite{buchmannwilliams}.

\begin{proof}
Let $\varepsilon_0$ be the fundamental unit of $K$. Our first observation is that, if $\beta$ is any minimal element, then so is $\pm\varepsilon_0^j\beta$ for $j\in\ZZ$. So, the set of minimal elements is the set of all associates (unit multiples) of minimal elements on the interval $[1,\varepsilon_0)$. Let these elements be denoted $1=\beta_0<\cdots<\beta_{l-1}$. We know there are only finitely many, for a lattice can only intersect a compact region (the closure of $R_{\varepsilon_0,1}$) in finitely many points. Then the minimal sequence is of the form:

\begin{equation*}
(1=\beta_0,\ldots,\beta_{l-1},\varepsilon_0,\ldots,\varepsilon_0\beta_{l-1},\varepsilon_0^2,\ldots).
\end{equation*}

This sequence has the property that $\beta_{i+l}=\varepsilon_0\beta_i$, and taking norms, this gives us that $\N_{i+l}=\N_i$. Thus, we have periodicity. Furthermore, we know that $\N_0=1=\N_{l}=\N(\varepsilon_0)$, and since $\varepsilon_0$ is the fundamental unit, we know that $\N_i>1$ for any positive $i<l$. This gives us that the period of the norm sequence is precisely $l$.
\end{proof}

Now, let $\mathcal{M}$ be the set of elements in the minimal sequence of $\alpha$ on the interval $[1,\varepsilon_0)$, and let $\mathcal{R}$ be the set of reduced principal ideals in $K$. We construct functions $F:\mathcal{M}\to\mathcal{R}$ and $G:\mathcal{R}\to\mathcal{M}$, which we will show to be inverses. This will establish a bijection between our two sets.

First, let $\gamma$ be a minimal element of $\mathcal{O}_K$ satisfying $1\leq \gamma<\varepsilon_0$, and let $J$ be the fractional ideal generated principally by $\gamma^{-1}$. Since $J=(\gamma^{-1})\cdot\mathcal{O}_K$, then $1=\gamma^{-1}\gamma$ is minimal in $J$. Let $L$ be the least integer such that $I=(L)J=(\frac{L}{\gamma})$ is an integral ideal, which we note is primitive. Now, $L=L\cdot 1$ is minimal in $I$. Since $L$ is rational, then $I$ is reduced, and we set $F(\gamma)=I$.

In the other direction, let $I$ be a reduced principal ideal. Then $I=(\eta)$ for some integer $\eta>0$. Since $I$ is reduced, we have that $L=\Len(I)$ is minimal in $(\eta)$. Then $\widehat{\gamma}=L\eta^{-1}$, must be minimal in $(\eta^{-1})(\eta)=\mathcal{O}_K$. Let $j=-\lfloor\log_{\varepsilon_0}\widehat{\gamma}\rfloor$, and let $\gamma=\varepsilon_0^j\widehat{\gamma}$. Then $\gamma$ is a minimal element in $\mathcal{O}_K$ satisfying $1\leq\gamma<\varepsilon_0$, so we set $G(I)=\gamma$.

Since the ideal $I$ could be written as a principal ideal in more than one way, we need to check that $G$ is well-defined. However, if $I=(\eta')$, then we know that $\eta'=\eta\varepsilon_0^r$ for some integer $r$. Thus, in the above argument, we obtain a $\widehat{\gamma}'$ that is an associate of $\widehat{\gamma}$, and therefore an associate of the same $\gamma$. So, $G$ is well-defined.

\begin{theorem}
\label{bijection}
The functions $F$ and $G$ defined above are inverses, providing a bijection between the sets $\mathcal{M}$ and $\mathcal{R}$.
\end{theorem}

We note that this result mirrors Proposition 4.3 from \cite{buchmannwilliams}.

\begin{proof}
First, we calculate $F(G(I))$, where $I=(\eta)$ is a reduced principal ideal with length $L$. We have that $G(I)=\gamma$ where $\gamma$ is some associate of $\widehat{\gamma}=\frac{L}{\eta}$. To apply $F$, we must choose the smallest integer $L'$ such that $\left(L'\gamma^{-1}\right)$ in an integer ideal. We know that $I=(\eta)=\left(L\widehat{\gamma}^{-1}\right)=\left(L\gamma^{-1}\right)$ is an integer ideal, and furthermore, a primitive one because it is reduced. If $L'<L$, then $I$ would not be primitive, so we have $L'=L$, and

\begin{equation*}
F(G(I))=F(\gamma)=\left(\frac{L}{\gamma}\right)=\left(\frac{L}{\widehat{\gamma}}\right)=(\eta)=I,
\end{equation*}

as desired.

In the other direction, we consider $G(F(\gamma))$, where $\gamma$ is a minimal element in $\mathcal{O}_K$ satisfying $1\leq\gamma<\varepsilon_0$. Let $L$ be the least positive integer such that $I=\left(\frac{L}{\gamma}\right)$ is an integer ideal. Then:

\begin{equation*}
G(F(\gamma))=G\left(\left(\frac{L}{\gamma}\right)\right)=\varepsilon_0^j\frac{\Len(L/\gamma)}{L/\gamma}=\varepsilon_0^j\frac{L}{L/\gamma}=\varepsilon_0^j\gamma,
\end{equation*}

where $j=-\left\lfloor\log_{\varepsilon_0}\gamma\right\rfloor=0$. This completes our proof.
\end{proof}

The above result seems to scratch the surface of a theory of generalized continued fractions, i.e., sequences that are sensitive to the structure of cubic fields, analogously as ordinary continued fractions are sensitive to the structure of quadratic fields. This appears to be a possible area for further research.

\section{Appendix 1: Proof of lemma on Z-modules}

The following is a proof of Lemma \ref{modulelemma}:

\begin{replemma}{modulelemma}
Let $M=[u_1,\ldots,u_n]$ be a free $\ZZ$-module of rank $n$. Let $M'\subseteq M$ be a submodule of full rank. Then we can write $M'=[a_{1,1}u_1, \ldots, a_{n,1}u_1+\cdots+a_{n,n}u_n]$, with all coefficients integral. Furthermore, we can suppose without loss of generality that, for $i=1,\ldots,n$, we have $a_{i,i}$ strictly positive, and for $j=i+1,\ldots,n$ we have $0\leq a_{j,i}<a_{i,i}$. Subject to these conditions, all $\frac{n(n+1)}{2}$coefficients are uniquely determined.
\end{replemma}

\begin{proof}
Let $M=[u_1,\ldots,u_n]$, and let $M'\subseteq M$ have full rank. We note that, if $n=1$, there is nothing to show, and we proceed by induction on $n$.

Let $\widetilde{M}=[u_1,\ldots,u_{n-1}]$; then $M'\cap\widetilde{M}$ is a submodule of $\widetilde{M}$ with full rank. By induction, we have that $M'\cap\widetilde{M}=[w_1,\ldots,w_{n-1}]$ with each $w_i=a_{i,1}u_1 + \cdots + a_{i,i}u_i$, all coefficients $a_{i,j}$ integral. Furthermore, for $i=1,\ldots,n-1$, we have $a_{i,i}>0$, and for $j=i+1,\ldots,n-1$, we have $0\leq a_{j,i}<a_{i,i}$.

Now, we define the set $I=\{k\in\ZZ : ku_n\in M'\oplus u_1\ZZ\oplus\cdots\oplus u_{n-1}\ZZ\}$. We observe that $I$ is a non-zero ideal of $\ZZ$, so put $I=(a_{n,n})$. By the definition of $I$, we have integers $b_i\in\ZZ$, for $i=1,\ldots,n-1$, such that $\widehat{w}_n=b_1u_1 + \cdots + b_{n-1}u_{n-1} + a_{n,n}u_n \in M'$.

Using the division algorithm repeatedly, we can write:

\begin{align*}
b_{n-1} &= q_{n-1}a_{n-1,n-1} + a_{n,n-1} \\
b_{n-2} - q_{n-1}a_{n-1,n-2} &= q_{n-2}a_{n-2,n-2} + a_{n,n-2} \\
&\vdots \\
b_1 - q_{n-1}a_{n-1,1} - \cdots - q_2a_{2,1} &= q_1a_{1,1} + a_{n,1}
\end{align*}

We thus obtain an element of $M'$:

\begin{align*}
w_n &= \widehat{w}_n - q_{n-1}w_{n-1} - \cdots - 	q_1w_1 \\
&= a_{n,1}u_1 + \cdots + a_{n,n}u_n,\\
\end{align*}

with coefficients satisfying the required conditions. We must now show that the set $\{w_1,\ldots,w_n\}$ spans $M'$.

It is clear that $N=[w_1,\ldots,w_n]\subseteq M'$. For the reverse inclusion, take an element $m\in M'$, and write $m=k_1u_1+\cdots+k_nu_n$ in terms of our original integral basis. Now, $k_n\in I$, so $k_n=t_na_{n,n}$ for some integer $t_n$. Subtracting $m-t_nw_n$, we obtain $(k_1-t_na_{n,1})u_1 + \cdots + (k_{n-1}-t_na_{n,n-1})u_n-1$, an element of $M'\cap\widetilde{M}$.

By the induction hypothesis, this element can be written $m-t_nw_n = t_1w_1 + \cdots t_{n-1}w_{n-1}$, which puts:

\begin{equation*}
m = t_1w_1 + \cdots + t_nw_n \in N,
\end{equation*}

as desired.

To see that the expression is unique subject to our constraints, supppose that $M'$ is also given by $M'=[w_1',\ldots,w_n']$, with $w_1'=a_{1,1}'u_1, \ldots, w_n'=a_{n,1}'u_1 + \cdots + a_{n,n}'u_n$, and that the positivity and bounding constraints are satisfied by these coefficients. Examining the differences $w_i-w_i'\in M'$, we see that all coefficients must match, proving uniqueness. 
\end{proof}

\section{Appendix 2: Python code for listing ideals}

The following algorithm can list every primitive ideal up to the Minkowski bound, or it can list every reduced ideal, in a given pure cubic field. Each ideal is listed as an ordered sextuple $(a,b,c,d,e,f)$, where the entries are the coefficients of the ideal's canonical form.

\begin{lstlisting}
import math
from fractions import gcd

def kPart(n):
    #Input=integer
    #Output=largest integer whose square divides n
    kPart=int(math.floor(math.sqrt(n)))
    success=0
    while (kPart>1)*(success==0):
        if n%kPart**2!=0:
            kPart=kPart-1
        else:
            success=1
    return kPart

def cubepart(m):
    #Input=integer
    #Output=2-vector:
    #1st entry=cube-free part of m
    #2nd entry=largest integer whose cube divides m
    div=math.floor(math.exp(math.log(m)/3))
    success=0
    while (div>1)*(success==0):
        if m%div**3!=0:
            div=div-1
        else:
            success=1
    return(m//div**3,div)

def isIdeal(a,b,c,d,e,f,m):
    k=kPart(m)
    h=m//k**2
    N=a*c*f
    area=c*f
    sigma=1
    if m**2%9==1:
        sigma=3
    pm=1
    if (m%9==8):
        pm=0-1
    p=(h*k-pm*k**3)//sigma
    q=(k-k**3)//sigma
    r=(k**2-pm*2*h+1)//sigma**2
    s=(h-pm*k**4)//sigma**2
    t=(k**3+2*k)//sigma
    ideal=1
    if (d*f+q*f**2-sigma*k*e**2-pm*2*k**2*e*f)%area!=0:
        ideal=0
    else:
        if (q*e*f+s*f**2-d*e-pm*k**2*e**2-t*e*f)%area!=0:
            ideal=0
        else:
            if ((k**2*c**2 + b**2)*f - pm*k**2*b*c*f - sigma*k*c*(b*e - c*d))%(N)!=0:
                ideal=0
            else:
                if (c*f*(p*c-q*b)+(b+pm*k**2*c)*(b*e-c*d))%(N)!=0:
                    ideal=0
                else:
                    if (f*(p*c*f-k**2*c*e-b*d-q*b*f)+pm*k**2*f*(2*b*e-c*d)+sigma*k*e*(b*e-c*d))%(N)!=0:
                        ideal=0
                    else:
                        if (f*(p*c*e-r*k**2*c*f-q*b*e-s*b*f)+(d+t*f+pm*k**2*e)*(b*e-c*d))%(N)!=0:
                            ideal=0
    return ideal

def isPrimitive(a,b,c,d,e,f):
    primitive=1
    divisor=2
    while (divisor<=c)*(divisor<=f)*(primitive==1):
        if (a%divisor==0)*(b%divisor==0)*(c%divisor==0)*(d%divisor==0)*(e%divisor==0)*(f%divisor==0):
            primitive=0
        else:
            divisor=divisor+1
    return primitive

def isReduced(a,b,c,d,e,f,m):
    isReduced=1
    k=kPart(m)
    h=m//k**2
    sigma=1
    pm=1
    if m**2%9==1:
        sigma=3
        if m%9==8:
            pm=0-1
    alpha=math.exp(math.log(m)/3)
    alphahat=math.exp(math.log(k*h**2)/3)
    if a>min(alpha,alphahat/sigma):
        z=0
        maxZ=sigma*a/alphahat
        while (z<maxZ)*(isReduced==1):
            yMin=(alphahat/alpha - pm*k)*z/sigma - 2*a/(math.sqrt(3)*alpha)
            yMax=(sigma*a - alphahat*z + math.sqrt((sigma*a-alphahat*z)*(sigma*a+3*alphahat*z)))/(2*sigma*alpha) - pm*k*z/sigma
            y=math.ceil(yMin)
            while (y-e*z//f)%c!=0:
                y=y+1
            while (y<yMax)*(isReduced==1):
                if ((y**2+z**2)!=0):
                    P1=(0-sigma*a-k*z-(alpha*(sigma*y+pm*k*z)+alphahat*z))/sigma
                    P2=(sigma*alpha*y+pm*alpha*k*z+alphahat*z-2*k*z-math.sqrt((2*sigma*a)**2-3*(sigma*alpha*y+pm*alpha*k*z-alphahat*z)**2))/(2*sigma)
                    P=max(P1,P2)
                    Q1=(sigma*a-k*z-(alpha*(sigma*y+pm*k*z)+alphahat*z))/sigma
                    Q2=(sigma*alpha*y+pm*alpha*k*z+alphahat*z-2*k*z+math.sqrt((2*sigma*a)**2-3*(sigma*alpha*y+pm*alpha*k*z-alphahat*z)**2))/(2*sigma)
                    Q=min(Q1,Q2)
                    if math.floor((Q-(y*b*f+z*c*d-z*b*e)/(c*f))/a) >= (P-(y*b*f+z*c*d-z*b*e)/(c*f))/a:
                        isReduced=0
                y=y+c
            z=z+f
    return isReduced

m=int(input('Starting m: '))
maxM=int(input('Maximum for m: '))
OutputFlag=int(input('enter 0 for all primitive ideals within Minkowski range; 1 for reduced ideals only:'))
while m<=maxM:
    mValid=1
    print()
    if cubepart(m)[0]==1:
        print("m =",m, " is a perfect cube")
        mValid=0
    if (cubepart(m)[1]>1)*(mValid==1):
        mPrime=cubepart(m)[0]
        k=kPart(mPrime)
        h=mPrime//k**2
        print("m =",m, " is not cube-free. See m =", min(h**2*k,h*k**2))
        mValid=0
    k=kPart(m)
    h=m//k**2
    if (h<k)*(mValid==1):
        print("m =",m, " is redundant with m =",h**2*k)
        mValid=0
    if mValid==1:
        sigma=1
        pm=1
        if m**2%9==1:
            sigma=3
            if (m%9==8):
                pm=0-1
        yn=(sigma-1)//2
        print("m =",m,", sigma =",sigma,", pm =",pm,", k =",k)
        maxL=math.floor(6*math.sqrt(3)*m/math.pi)
        #That's the Minkowski bound, based on the volume of the region R
        Mink2DConst=(9+2*math.sqrt(3)*math.pi)/36
        #For use below, when applying Minkowski's convex body theorem to R intersect {z=0}
        print("maxL=",maxL)
        a=1
        while a<=maxL:
            #c=1
            c=math.ceil(a*Mink2DConst/(math.exp(math.log(m)/3)))
            #for c less than this bound, there is an ideal elemnt inside R for z=0 by Minkowski's convex body theorem
            #Indeed, the area of R intersect {z=0} is a^2(9+2sqrt(3)pi)/(9m^(1/3)), so the ideal fails to be reduced
            #whenever ac is less than one fourth of that, or when c<a(9+2sqrt(3)pi/(36m^(1/3))
            while (c<=a):
                if(a%c==0):
                    b=0
                    while b<a:
                        d=0
                        while d<a:
                            e=0
                            while e<c:
                                if k==1:
                                    fDiv=1
                                else:
                                    fDiv=abs(gcd(gcd(gcd(gcd(gcd(gcd(a,sigma*k),a*e//c),b+yn*pm*k**2*c),d+yn*pm*k**2*e),b*e//c-yn*d),b*e//c+yn*pm*k**2*e))
                                f=math.ceil(math.pi*a**2/(6*math.sqrt(3)*m*c))
                                #for f less than this bound, there is an ideal element inside R by Minkowski's convex body theorem
                                while (f<=fDiv):
                                    if (fDiv%f==0):
                                        if isIdeal(a,b,c,d,e,f,m):
                                            if isPrimitive(a,b,c,d,e,f):
                                                red=isReduced(a,b,c,d,e,f,m)
                                                if OutputFlag==0:
                                                    #print()
                                                    if red==1:
                                                        print("Primitive Ideal (",a,b,c,d,e,f,"), with N =",a*c*f," is a reduced ideal.")
                                                    else:
                                                        print("Primitive Ideal (",a,b,c,d,e,f,"), with N =",a*c*f," is not reduced.")
                                                else:
                                                    if red==1:
                                                        #print()
                                                        print("Reduced ideal: (",a,b,c,d,e,f,") has norm N =",a*c*f)
                                    f=f+1
                                e=e+1
                            d=d+1
                        b=b+c
                c=c+1
            a=a+1
    m=m+1 

\end{lstlisting}

\section{Appendix 3: Python code for generating norm sequences}

The following algorithm takes a cube-free integer $m$ as input, as well as an upper bound on $z$-values, and lists elements of the norm sequence for $\sqrt[3]{m}$ that have $\phi$-images with $z$-coordinates less than the upper bound.

\begin{lstlisting}
import math

def cubepart(m):
    #Input=integer
    #Output=2-vector:
    #1st entry=cube-free part of m
    #2nd entry=largest integer whose cube divides m
    div=math.floor(math.exp(math.log(m)/3))
    success=0
    while (div>1)*(success==0):
        if m%div**3!=0:
            div=div-1
        else:
            success=1
    return(m//div**3,div)

def kPart(n):
    #Input=integer
    #Output=largest ineger whose square divides n
    kPart=int(math.floor(math.sqrt(n)))
    success=0
    while (kPart>1)*(success==0):
        if n%kPart**2!=0:
            kPart=kPart-1
        else:
            success=1
    return kPart

def Value(x,y,z):
    xtilde=x+z*k/sigma
    ytilde=y+pm*z*k/sigma
    ztilde=z/sigma    
    Val=xtilde+alpha*ytilde+alphahat*ztilde
    return Val

def Shadow(x,y,z):
    xtilde=x+z*k/sigma
    ytilde=y+pm*z*k/sigma
    ztilde=z/sigma
    Sh=xtilde**2+alpha**2*ytilde**2+alphahat**2*ztilde**2-alpha*xtilde*ytilde-alphahat*xtilde*ztilde-alpha*alphahat*ytilde*ztilde
    return Sh

def Norm(x,y,z):
    xtilde=x+z*k/sigma
    ytilde=y+pm*z*k/sigma
    ztilde=z/sigma
    N=xtilde**3+h*k**2*ytilde**3+h**2*k*ztilde**3-3*h*k*xtilde*ytilde*ztilde
    return N
m=1
m=int(input('m:'))
while m!=0:
    mValid=1
    print()
    if cubepart(m)[0]==1:
        print("m =",m, " is a perfect cube")
        mValid=0
    if (cubepart(m)[1]>1)*(mValid==1):
        mPrime=cubepart(m)[0]
        k=kPart(mPrime)
        h=mPrime//k**2
        print("m =",m, " is not cube-free. See m =", min(h**2*k,h*k**2))
        mValid=0
    k=kPart(m)
    h=m//k**2
    if (h<k)*(mValid==1):
        print("m =",m, " is redundant with m =",h**2*k)
        mValid=0
    if mValid==1:
        sigma=1
        pm=1
        if m**2%9==1:
            sigma=3
            if (m%9==8):
                pm=0-1
        yn=(sigma-1)//2
        print("m =",m,", sigma =",sigma,", pm =",pm,", k =",k)
        maxZ=int(input('maximum z-value:'))
        alpha=math.exp(math.log(m)/3)
        alphahat=math.exp(math.log(h**2*k)/3)
        a=1
        b=0
        c=0
        bestList=[[1,0,0]]
        Val=1
        Sh=1
        N=Val*Sh
        print(bestList[0],",   Val=",Val,", Sh=",Sh,", N=",N)
        c=c+1
        while c<maxZ:
            currentBest=[]
            x0=math.floor((alphahat-k)/sigma*c)
            y0=math.floor((alphahat-pm*alpha*k)/alpha/sigma*c)
            if Shadow(x0,y0,c)<Sh:
                currentBest.append([x0,y0,c])
            if Shadow(x0,y0+1,c)<Sh:
                currentBest.append([x0,y0+1,c])
            if Shadow(x0+1,y0,c)<Sh:
                currentBest.append([x0+1,y0,c])
            if Shadow(x0+1,y0+1,c)<Sh:
                currentBest.append([x0+1,y0+1,c])
            currentBest.sort(key=lambda x:Value(x[0],x[1],x[2]))
            while len(currentBest)>0:
                if Shadow(currentBest[0][0],currentBest[0][1],currentBest[0][2])<Sh:
                    newBest=currentBest.pop(0)
                    Val=Value(newBest[0],newBest[1],newBest[2])
                    Sh=Shadow(newBest[0],newBest[1],newBest[2])
                    N=Norm(int(newBest[0]),int(newBest[1]),int(newBest[2]))
                    print(newBest,",   Val=",Val,", Sh=",Sh,", N=",N)
                    bestList.append(newBest)
                else:
                    currentBest.remove(currentBest[0])
            c=c+1
    m=int(input('m:'))
\end{lstlisting}

\section{Questions for further research}

\begin{itemize}
	\item Can we obtain analogous results for arbitrary cubic fields with negative discriminant? For quartic fields with rank-$1$ unit groups?
	\item Can we further extend the analogy between norm sequences and the continued fraction algorithm?
	\item What is the corresponding structure like when the rank of the unit group is greater than $1$?
\end{itemize}

\bibliographystyle{amsplain}

\bibliography{bibliography}

\end{document}